\documentclass[11pt]{article}
\usepackage{amssymb,amsmath}
\usepackage{latexsym,bm}
\usepackage{dsfont}

\usepackage{mathrsfs}
\newtheorem{theorem}{Theorem}
\newtheorem{corollary}[theorem]{Corollary}

\newtheorem{problem}{Problem}
\newtheorem{lemma}[theorem]{Lemma}

\newenvironment {proof} {{\it
Proof.}}{\hspace*{\fill}$\Box$\par\vspace{4mm}}
\usepackage{amsmath}

\newcommand{\bea}{\begin{eqnarray*}}
\newcommand{\eea}{\end{eqnarray*}}
\newcommand{\be}{\begin{equation}}
\newcommand{\ee}{\end{equation}}
\newcommand{\ben}{\begin{eqnarray*}}
\newcommand{\een}{\end{eqnarray*}}

\newcommand{\G}{{\mathscr{G}}}
\date{}
\begin{document}
\title{Linear maps  on nonnegative symmetric  matrices  preserving the independence number}
\author{ Yanan Hu\thanks{Department of Mathematics, East China Normal University, Shanghai 200241, China. (Email: huyanan530@163.com) }, Zejun Huang,\thanks{Corresponding author. College of Mathematics and Statistics, Shenzhen University, Shenzhen 518060, China.  (Email: mathzejun@gmail.com) }  }
\maketitle
\begin{abstract}
 The independence number of a square matrix $A$, denoted by $\alpha(A)$, is the maximum order of its principal zero submatrices. Let $S_n^{+}$ be the set of $n\times n$ nonnegative symmetric matrices with zero trace. Denote by $J_n$  the $n\times n$ matrix with all entries equal to one.  Given any integer $n$, we prove that a   linear map $\phi: S_n^+\rightarrow S_n^+$  satisfies
$$\alpha(\phi(X))= \alpha(X)  {\quad\rm for~ all\quad}X\in S_n^+$$
if and only if there is a permutation matrix $P$   such that
$$\phi(X)=H\circ(P^TXP)\quad { \rm for~ all\quad}X\in S_n^+,$$
where $H=\phi(J_n-I_n)$ with all off-diagonal entries positive.

\end{abstract}


\section{Introduction and main result}

Linear preserver problems on matrices and operators have a long history,  which concern the characterization of linear maps on matrices or operators preserving special
properties.  In 1897, Frobenius \cite{F} showed that a linear operator
$\phi: M_n\rightarrow M_n$ satisfies
$$\det(\phi(A)) = \det(A) \qquad \hbox{ for all } A \in M_n$$
if and only if  there are $M, N\in M_n$ with $\det(MN) = 1$ such that
$\phi$ has the form
\begin{equation*} \label{standard}
A \mapsto MAN \quad \hbox{ or } \quad A \mapsto MA^tN,
\end{equation*}
where $M_n$ denotes the set of $n\times n$ complex matrices. Since then, lots of  linear preserver problems have been investigated.     There are many new directions and active research on preserver
problems motivated by theory and applications; see \cite{FHLS, FLPS, LP, M7, W}.  Particularly, Minc \cite{HM}  characterized those linear transformations that map nonnegative matrices into nonnegative matrices and preserve the spectrum of each nonnegative matrix. In \cite{BS1, BS2, BS3, H}, linear maps preserving certain combinatorial properties of matrices are studied.

The {\it independence number} of a square matrix $A$, denoted by $\alpha(A)$, is the maximum order of its principal zero submatrices.
Let $S_n^{+}$ be the set of $n\times n$ nonnegative symmetric matrices with zero trace.  In this paper, we study the following problem.

\begin{problem}\label{p2}
	Let $n$  be  a positive integer. Characterize the linear map $\phi: S_n^+\rightarrow S_n^+$ such that
	\begin{equation}\label{eqhm1}
	\alpha(\phi(X))= \alpha(X) {~~~~\rm for~ all\quad }X\in S_n^+,
	\end{equation}
\end{problem}

We solve Problem 1 in this paper. Let $J_n$ be the $n\times n$ matrix with all entries equal to one.  Denote by $A\circ B$ the Hadamard product of two matrices $A$ and $B$. Our main result states as follows.

\begin{theorem}\label{mth1}
	Let $n$   be a positive integer.  Then  a linear map $\phi: S_n^+\rightarrow S_n^+$  satisfies    (\ref{eqhm1}) if and only if there is a permutation matrix $P$   such that
	$$\phi(X)=H\circ(P^TXP)\quad { \rm for~ all\quad}X\in S_n^+,$$
where $H=\phi(J_n-I_n)$ with all off-diagonal entries positive.
\end{theorem}

 Given a matrix $A=(a_{ij})\in S^+_n$, we define its graph $G(A)$ by the graph with vertex set $\langle n\rangle=\{1,2,\ldots,n\}$ and edge set $\{(i,j): a_{ij}\ne 0\}$. Conversely, we can define the adjacency matrix of a graph $G$, which is denoted by $A(G)$. Our strategy is to transform Problem 1 to a linear preserver problem on graphs by considering the graphs of matrices in $S^+_n$. We will   study a linear preserver problem on graphs in Section 2,  and then present the proof of Theorem \ref{mth1} in Section 3.

\section{Linear maps on graphs preserving a given independence number }

In this section, we study a linear preserver problem on graphs.
Graphs in this paper are simple.  We denote by $\langle n\rangle=\{1,2,\ldots,n\}$ and by  $\G_n$ the set of graphs with  vertex set $\langle n\rangle$. For a graph $G$, we always denote its vertex set by $V(G)$ and denote its edge set by $E(G)$.  The {\it order} of a graph is its number of vertices and the {\it size} of a graph is its number of edges.

  For two graphs $G_1,G_2\in \G_n$, its {\it union} $G_1\cup G_2$ is the graph with vertex set $\langle n\rangle$ and edge set $E(G_1)\cup E(G_2)$.
 If $V(G_1)\subseteq V(G_2)$ and $E(G_1)\subseteq E(G_2)$, then $G_1$ is said to be a {\it subgraph} of $G_2$, which is denoted as $G_1\subseteq G_2$.
  An {\it independent set} of a graph $G$ is a subset of $V(G)$ such that no two vertices in the subset are adjacent.  An independent set of a graph $G$ is said to be {\it maximum} if $G$ has no independent set with more vertices.
  The {\it independence number} of a graph $G$, denoted by $\alpha(G)$, is  the cardinality of  a maximum independent  set of $G$.

  A map $\phi: \G_n\rightarrow \G_n$ is said to be {\it linear} if
  $$\phi(G_1\cup G_2)=\phi(G_1)\cup \phi(G_2)~~{\rm for ~all~}G_1,G_2\in \G_n.$$
 Moreover, if $\phi$ maps the complete graph $K_n$ to itself, i.e. $\phi(K_n)=K_n$, then we call $\phi$ a {\it complete linear map}.

 Linear maps on graphs were introduced by Hershkowitz in \cite {H}, in which he characterized the  linear maps that map the set of all graphs
 which contain no cycle of length greater than or equal to $k$ into or onto
 itself.
 Following  Hershkowitz's direction, we investigate the following problem on graphs in this section.

\begin{problem}\label{p1}
Given positive integers $n$ and $t$, determine the  complete linear maps $\phi: \G_n\rightarrow \G_n$ such that
 \begin{equation}\label{h1}
 \alpha(\phi(G))=t\quad { if~and ~only ~if\quad } \alpha(G)=t {\quad\rm for~ all\quad}G\in \G_n.
 \end{equation}
\end{problem}

Note that the only graph in $\G_n$ with   independence number  $n$ is the empty graph. The above problem is trivial when $t=n$. So we only consider the case $1\le t\le n-1$.

Given $i,j\in\langle n\rangle$, denote by $G_{ij}$ the graph in $\G_n$ with edge set $\{(i,j)\}$. Let $\phi:\G_n\rightarrow\G_n$  be a linear map. If there is a permutation $\sigma$ of $\langle n\rangle $ such that
$$\phi(G_{ij})=G_{\sigma(i)\sigma(j)}~{\rm for~all~}i,j\in \langle n\rangle,$$
then $\phi$ is said to be a {\it vertex permutation};
if $\phi(K_n)=K_n$ and each $\phi(G_{ij})$ contains exactly one edge for all distinct  $i,j\in \langle n\rangle$,
 then $\phi$ is said to be an {\it edge permutation}.

We have the following solution to Problem 2.
\begin{theorem}\label{th1}
	Let $n,t$ be positive integers such that $ t\le n-1$.  Then a complete linear map $\phi: \G_n\rightarrow \G_n$  satisfies (\ref{h1})   if and only if one of the following holds.
	\begin{itemize}
		\item [(i)] $t=1$ and $\phi$ is an edge permutation.
		\item [(ii)] $t\ge 2$ and $\phi$ is a vertex permutation.
	\end{itemize}
\end{theorem}

To prove Theorem \ref{th1}, we first present some elementary properties on the independence number  of graphs, which are of independent interests.  For two sets $V_1, V_2$,   we define
$$V_1-V_2\equiv\{x: x\in V_1, x\not\in V_2\}.$$  We denote by $| V|$ the cardinality of a set  $V$.
The following lemma can be deduced from Lemma 2.1 of \cite {LM}. For completeness, we give a new proof.

 \begin{lemma}\label{leh1}
 	Let $G\in \G_n$ with independence number  $\alpha$. Then the maximum independent sets of $G$ have at least $2\alpha-n$ common vertices.
 \end{lemma}
 \begin{proof}
 	Suppose $\{V_1,V_2,\ldots,V_k\}$ is the set of all the maximum independent sets of $G$.  Then
 	$$|V_i|= \alpha {\rm~ for~} i=1,\ldots,k.$$
 	Let
 	$$q_i=|V_1\cap V_2\cap\cdots\cap V_i|  {\rm ~~for~} i=2,\ldots,k$$  and
 	$$p_i=|V_i-  V_1- V_2-\cdots-  V_{i-1}|{\rm ~~for~} i=3,\ldots,k.$$
 It suffices to prove  $q_k\geq2\alpha-n$.
 	
 	We claim that
 	$$p_i\geq q_{i-1}-q_i {\rm ~for~} i=3,\ldots,k.$$
 	Otherwise, suppose $p_i<q_{i-1}-q_i$ for some $i\in \{3,\ldots,k\}$. Let
 	$$S=V_1\cap V_2\cap \cdots\cap V_{i-1}-V_i {\rm ~and~} T=V_i\cap (V_1\cup V_2\cup \cdots\cup V_{i-1}).$$
 	Then $|S|=q_{i-1}-q_i$, $|T|=|V_i|-p_i$ and $S\cup T$ is an independent set of $G$.  Moreover, we have
 	$$|S\cup T|=|S|+|T|=|V_i|-p_i+q_{i-1}-q_i>\alpha,$$
 	a contradiction.
 	
 	It follows that
 	\begin{eqnarray*}
 		n&\geq&|V_1\cup V_2\cup\cdots\cup V_k|\\
 		&=&|V_1|+|V_2-V_1|+|V_3-V_1-V_2|+\cdots+|V_k-\bigcup_{j=1}^{k-1}V_j|\\
 		&=&\alpha+\alpha-q_2+p_3+\cdots+p_k\\
 		&\geq& 2\alpha-q_2+(q_2-q_3)+\cdots+(q_{k-2}-q_{k-1})+(q_{k-1}-q_k)\\
 		&=&2\alpha-q_k.
 	\end{eqnarray*}
 	Hence,  $q_k\geq2\alpha-n$.
 \end{proof}

\begin{lemma}\label{le31}
	Let $G_1,G_2\in \G_n$. Then
	$$\alpha(G_1\cup G_2)\ge \alpha(G_1)+\alpha(G_2)-n.$$
\end{lemma}
\begin{proof}
Suppose $V_1$ and $V_2$ are maximum independent sets of $G_1$ and $G_2$, respectively. Then $V_1\cap V_2$ is an independent set of $G_1\cup G_2$ and
$$\alpha(G_1\cup G_2)\ge |V_1\cap V_2|=|V_1|+|V_2|-|V_1\cup V_2|\ge \alpha(G_1)+\alpha(G_2)-n.$$
\end{proof}

Especially,  let $G'$ be a graph obtained by adding an edge to a graph $G$. Then
 \begin{equation}\label{eqh1}
 \alpha(G')=\alpha(G)~{\rm or~}\alpha(G')=\alpha(G)-1.
 \end{equation}

The (unlabled) {\it Tur\'an graph} $T(n,r)$  is a complete $r$-partite graph on $n$ vertices in which each vertex partition has size $ \lceil n/r\rceil$ or $\lfloor n/r\rfloor$.  The {\it complement} of a graph $G$  is the graph  with vertex set $V(G)$ such that two distinct vertices are adjacent if and only if they are not adjacent in $G$. We denote by $T'(n,r)$ the complement of  $T(n,r)$,  which is the union of  $r$ complete graphs whose orders are  $ \lceil n/r\rceil$ or $\lfloor n/r\rfloor$.

An {\it isomorphism} of graphs $G$ and $H$ is a bijection $f: V(G)\rightarrow V(H)$ such that
$$(u,v)\in E(G) {\rm~if~ and ~only~ if~} (f(u),f(v))\in E(H)~{\rm ~for~ all~} u,v\in V(G).$$  If an isomorphism exists between two graphs $G$ and $H$, then $G$ and $H$  are said to be {\it isomorphic} and denoted as $G\cong H$.
The following result is well known \cite{TP}.
 \begin{lemma}[Tur\'an]\label{tu1}
 If $G\in \G_n$ is $K_{r+1}$-free with $m$ edges, then $$m\leq (1-\frac{1}{r})\frac{n^2}{2}$$
with equality if and only if $G\cong T(n,r)$.
 \end{lemma}
 The {\it clique number} of a graph $G$, denoted by $\omega(G)$, is the maximum number of vertices in a complete subgraph of $G$.
Since the independence number  of a graph is the clique number of its complement, by Lemma \ref{tu1} we have the following corollary.
 \begin{corollary}\label{co1}
Let $G\in \G_n$  with $m$ edges and $\alpha(G)=r$. Then $$m\ge  \frac{n^2}{2r}-\frac{n}{2}$$
	with equality  if and only if $G\cong T'(n,r)$.
\end{corollary}

The following lemmas are also needed in the proof of Theorem \ref{th1}.
  \begin{lemma}\label{lehh1}
 	Let $n,t$ be integers such that  $1\le t\le n-1$. Suppose $\phi: \G_n\rightarrow \G_n$ is a complete linear map satisfying (\ref{h1}). Then $\phi(G_{ij})$ is not an empty graph  for all $i,j\in\langle n\rangle$ with $i\neq j.$
 \end{lemma}
 \begin{proof}
 	Suppose there exist  distinct $i,j\in \langle n\rangle$ such that $E(\phi(G_{ij}))= \emptyset$. Let $G\in\G_n$ be the smallest graph with independence number $t$ such that $(i,j)\in E(G)$. Then  $\alpha(\phi(G))=t$. Let $G'=G-G_{ij}$.  Then $\alpha(G')=t+1$ and $\alpha(\phi(G'))=\alpha(\phi(G))=t$, a contradiction.
 	Hence,
 	$$E(\phi(G_{ij}))\neq \emptyset {\rm ~ for ~all~} i,j\in\langle n\rangle,i\neq j.$$
 \end{proof}
Two distinct edges are said to be {\it  adjacent} if they share a common vertex. Otherwise  they are said to be {\it separate}.
\begin{lemma}\label{hle2}
		Let $n,t$ be integers such that  $1\le t\le n-1$. Suppose $\phi: \G_n\rightarrow \G_n$ is a complete linear map satisfying (\ref{h1}).  Then $\phi(G_{ij})$ contains no separate edges for all $ i,j\in  \langle n\rangle$ with $i\ne j$.
 \end{lemma}
 \begin{proof}
 If $t=n-1$, then $G\in \G_n$ contains no separate edges  when $\alpha(G)=t$, and hence the result holds. Now we assume $1\le t\le n-2$.
 	
 Suppose there exist $i_1,j_1\in\langle n\rangle$ such that $\phi(G_{i_1j_1})$ contains two separate edges $(s_1,t_1)$ and $(s_2,t_2)$.  Without loss of generality, we may assume $(s_1,t_1)$ and $(s_2,t_2)$ are edges in a graph $H\cong T'(n,t)$,  since we can obtain $H$ by adding edges to $G_{s_1t_1} \cup G_{s_2t_2}$.

 Since $\phi$ is complete, for any $p,q\in \langle n \rangle$, there is a graph $G_{uv}$ such that $(p,q)\in E(\phi(G_{uv}))$.  Let $\gamma$ be the size of  $H$.
 We construct new graphs $G_k$ in $\phi(\G_n)$ for $k=1,\ldots, \gamma-1$ as follows.
 Take $G_1=\phi(G_{i_1j_1})$.    For $k\ge 2$, if $E(H)-E(G_{k-1})=\emptyset$, we set $G_k=G_{k-1}$. Otherwise, we choose $(s_{k+1},t_{k+1})\in E(H)-E(G_{k-1})$ and find a graph $G_{i_kj_k}$ such that $(s_{k+1},t_{k+1})\in E(\phi(G_{i_kj_k}))$. Let $G_k=G_{k-1}\cup \phi(G_{i_kj_k})$. Denote $G'=\bigcup_{p=1}^{\gamma-1}G_{i_pj_p}$. Then
 \begin{equation}\label{eqh21}
 	G_{\gamma-1}=\phi(G')
 	\end{equation}
 and $ H\subseteq G_{\gamma-1}$. By (\ref{eqh1}) we see that
 	$$\alpha(G_{\gamma-1})\le t{\rm ~and~} \alpha(G_1)\ge \alpha(G_2)\ge\cdots\ge \alpha(G_{\gamma-1}).$$
 	Moreover, since the size of $G'$ is less than $\gamma$, applying Corollary \ref{co1} we have
 	\begin{equation}\label{eqh22}
 	\alpha(G')>t.
 	\end{equation}

If $\alpha(G_{\gamma-1})= t$,
then (\ref{eqh21})   and (\ref{eqh22})   contradicts (\ref{h1}). If $\alpha(G_{\gamma-1})< t$, then we can add edges to $G'$ one by one to get a graph $H_1$ with  $\alpha(H_1)=t$ and  $\alpha(\phi(H_1))<t$, which contradicts (\ref{h1}).

Therefore, $\phi(G_{ij})$ contains no separate edges for all $ i,j\in  \langle n\rangle$ with $i\ne j$.
 \end{proof}

  If  a graph $G$ contains a vertex $u$  such that $E(G)=\{(u,x): x\in V(G)\setminus\{u\}\}$, then $G$ is said to be a {\it star} with {\it center} $u$.

\begin{lemma}\label{leh2}
	Let $n,t$ be integers such that  $1\le t\le n-1$. Suppose $\phi: \G_n\rightarrow \G_n$ is a complete linear map satisfying (\ref{h1}).  Then
	\begin{equation}\label{eq41}
	| E(\phi(G_{ij}))|=1~~ for~ all~  i,j\in  \langle n\rangle~with~i\ne j.
	\end{equation}
 \end{lemma}
 \begin{proof}
 	Since $\phi$ is complete, by Lemma \ref{lehh1} and Lemma \ref{hle2}, it suffices to prove that $\phi(G_{ij})$ contains no adjacent edges for all $ i,j\in  \langle n\rangle.$
 		If $ t\le\lceil\frac{n}{2}\rceil-1$, then $T'(n,t)$ contains adjacent edges. Applying the same arguments as in the proof of Lemma \ref{hle2} we can deduce the result.
 	
Now suppose $\lceil\frac{n}{2}\rceil\le t\le n-1$.
 By Lemma \ref{lehh1} and Lemma \ref{hle2}, we have  the following three possibilities:
 \begin{itemize}
 	\item[(i)] $\phi(G_{ij})$ is the union of a triangle and $n-3$ isolated vertices;
 	\item[(ii)] $\phi(G_{ij})$ is the union of a star   of order $\ge 3$ and some isolated vertices;
 	\item[(iii)] $\phi(G_{ij})$ is the union of an edge and $n-2$ isolated vertices.
 \end{itemize}

Suppose (i) holds. Then
\begin{equation}\label{eqh23}
\alpha(\phi(G_{ij}))=n-2{\rm ~ and~}  \alpha(G_{ij})=n-1.
\end{equation}
   Let $G_0=G_{ij}$. If $t=n-1$ or $n-2$, then (\ref{eqh23}) contradicts (\ref{h1}). Therefore, $$t\le n-3<\alpha(\phi(G_0)),$$ which implies $n\ge 6$. By Lemma \ref{leh1}, there are two vertices $u_1,v_1\in \langle n\rangle$ such that they are in all maximum independent sets of $\phi(G_0)$.  Since $\phi$ is complete, there exist $s_1,t_1\in \langle n\rangle$ such that  $G_{u_1v_1}\subset\phi(G_{s_1t_1})$.
Denote by $G_1=G_0\cup G_{s_1t_1}$. Then $$\alpha(\phi(G_1))<\alpha(\phi(G_0))=n-2 \le \alpha(G_1).$$
If $\alpha(\phi(G_1))=t$, then $\alpha(G_1)>t$ contradicts (\ref{h1}). If $\alpha(\phi(G_1))<t$, then  we may add new edges to $G_1$ to construct a graph $H_1$ such that
$$\alpha(H_1)=t {\rm~and ~}\alpha(\phi(H_1))\le\alpha(\phi(G_1))<t,$$ which also contradicts (\ref{h1}).
If $\alpha(\phi(G_1))>t$, then repeating the above process by adding new edges to $G_1$ we can deduce a contradiction. Therefore, (i) does not hold.

Suppose (ii) holds. Let $s$ be the center of $\phi(G_{ij}) $.  Firstly, we conclude that all edges of $\phi(G_{ik})$ and $\phi(G_{jk})$ are incident with $s$ for all $k\in \langle n\rangle\setminus\{i,j\}$, i.e., $\phi(G_{ik})$  and $\phi(G_{jk})$  are two  stars with center $s$.  Otherwise,  suppose an edge of $\phi(G_{ik})$ or $\phi(G_{jk})$ is not incident with $s$ for some $k\in \langle n\rangle\setminus\{i,j\}.$ Let
$G=G_{ij}\cup G_{ik}$ and $H=G_{ij}\cup G_{jk}$. Then by Lemma \ref{le31}, we have
\begin{equation}\label{eq411}
\alpha(\phi(G))=n-2, ~{ ~}\alpha(G)=n-1
\end{equation}or
\begin{equation}\label{eq412}
\alpha(\phi(H))=n-2, ~{ ~}\alpha(H)=n-1.
\end{equation}
If $t=n-1$ or $n-2$, then one of (\ref{eq411}) and (\ref{eq412}) contradicts (\ref{h1}).
If  $t<n-2$, then we can apply the same argument on $G$ or $H$ as in the previous paragraph to deduce a contradiction.

  Now let $G_1=G_{ij}\cup G_{ik}\cup G_{jk}$ with $k\in \langle n\rangle\setminus\{i,j\} $. Then  $\phi(G_1)$ is a star with center $s$ and
  $$\alpha(\phi(G_1))=n-1,\quad \alpha(G_1)=n-2.$$
  Applying Lemma \ref{leh1}, we can  add edges to $G_1$ one by one to
  obtain new graphs $G_2,\ldots, G_{n-t-1}$ such that
  $$\alpha(G_{u})=\alpha(G_{u-1})-1=n-u-1, {\rm ~for~}u=2,\ldots,n-t-1.$$
  On the other hand, by the previous arguments we see that    $\phi(G_{pq})$ is a star  for all $p,q\in \langle n\rangle$. Applying Lemma \ref{le31} we have
    $$ \alpha(\phi(G_{u}))\geq\alpha(\phi(G_{u-1}))-1{\rm ~for~}u=2,\ldots,n-t-1.$$
   Hence we have
   $$\alpha(G_{n-t-1})=t {\rm ~and~} \alpha(\phi(G_{n-t-1}))\geq t+1,$$
  a contradiction.  Hence, (ii) does not hold.

 Therefore,   $\phi(G_{ij})$ does not contain adjacent edges and we have (\ref{eq41}).
\end{proof}
From Lemma \ref{leh2} we have the following corollary.
\begin{corollary}\label{coh2}
Let $n,t$ be integers such that  $1\le t\le n-1$.  Then a complete linear map on $\G_n$ satisfying (\ref{h1}) is bijective.
\end{corollary}

\begin{lemma}\label{leh4}
 Let $n,t$ be integers such that  $2\le t\le n-1$. Suppose $\phi: \G_n\rightarrow \G_n$ is  a complete linear map satisfying (\ref{h1}).  Then for any distinct $i,j,k\in \langle n\rangle$, $E(\phi(G_{ij})\cup\phi(G_{ik}))$ consists of two adjacent edges.
 \end{lemma}
\begin{proof}
	Let $i,j,k\in\langle n\rangle$ be distinct.
 By Lemma \ref{leh2}, $\phi(G_{ij})=G_{ab}, \phi(G_{ik})=G_{cd}$ for some $a,b,c,d \in \langle n\rangle$.
 Suppose $(a,b)$ and $(c,d)$ are separate. If $t=n-1$ or $n-2$, then
 $$\alpha(\phi(G_{ij}\cup G_{ik}))=n-2~{\rm and~} \alpha(G_{ij}\cup G_{ik})=n-1$$
contradicts (\ref{h1}).

 If $2\le t\le n-3$, suppose $\phi(G_{jk})=G_{uv}$ with $u,v\in \langle n\rangle$.
Then we can construct a graph  $H\cong T'(n,t)$ by adding edges to $\phi(G_{ij})\cup \phi(G_{ik})$ such that $(u,v)\not\in H$. Recall that any  minimum graph of order $n$ with independence number $t$ is  isomorphic to $T'(n,t)$. Since
$(i,j),(i,k)\in E(\phi^{-1}(H))$, $(j,k)\not\in E(\phi^{-1}(H))$  and $|E(\phi^{-1}(H))|=|E(H)|=|E(T'(n,t))|$, we see that $\phi^{-1}(H)$ is not  isomorphic to $T'(n,t)$. Hence we have
$$\alpha(H)=t~{\rm~and~} \alpha(\phi^{-1}(H))\ne t$$
a contradiction.

Therefore, $(a,b)$ and $(c,d)$ are adjacent edges and $E(\phi(G_{ij})\cup(\phi(G_{ik}))$ consists of two adjacent edges.
\end{proof}
Using a similar proof with \cite[Poroposition 3.15]{H} we have the following.
\begin{lemma}\label{hle5}
	Let $n,t$ be integers such that  $2\le t\le n-1$. Suppose $\phi: \G_n\rightarrow \G_n$ is a complete linear map satisfying (\ref{h1}).  Then for every $i\in\langle n\rangle$, there exists $i'\in\langle n\rangle$ such that
\begin{equation}\label{eqh5}
\phi(\bigcup\limits_{j=1,j\ne i}^{n}G_{ij})=\bigcup\limits_{j=1,j\ne i'}^{n}G_{i'j}.
\end{equation}
 \end{lemma}
 \begin{proof}
 	By Lemma \ref{leh4}, the result is clear for $n=3$. Now we suppose $n\ge 4$.
Given any $i\in \langle n\rangle $, choose distinct  $j_1,j_2\in \langle n\rangle\setminus\{i\}$. By Lemma \ref{leh4}, there exist distinct $i',p,q\in\langle n\rangle$ such that
$$\phi(G_{ij_1})=G_{i'p}~{\rm and~}  \phi(G_{ij_2})=G_{i'q}.$$

Suppose there exists  $j_3\in\langle n\rangle\setminus\{i,j_1,j_2\}$ such that $\phi(G_{ij_3})=G_{uv}$ with $i'\notin\{u,v\}$. By Lemma \ref{leh4}, we have $G_{uv}=G_{pq}$.

If $n=4$, let $$G=G_{ij_1}\cup G_{ij_2}\cup G_{ij_3} {\rm ~ and~} H=G_{i'p}\cup G_{i'q}\cup G_{pq}.$$
Then $$\phi(G)=H, ~~\alpha(G)=3,~~\alpha(H)=2$$
contradicts (\ref{h1}) whenever $t=2$ or 3.

If $n\ge 5$, choose $j\in\langle n\rangle\setminus\{i,j_1,j_2,j_3\}$. By Lemma \ref{leh4}, we have $\phi(G_{ij})=G_{rs}$ with $(r,s)\in E(K_n)\setminus\{(i',p),(i',q),(p,q)\}$ being adjacent with $(i',p),(i',q)$ and $(p,q)$, which is impossible.

Therefore, for all $j\in \langle n\rangle\setminus \{i\}$, we have $ \phi(G_{ij})=G_{i'k}$ for some $k\in \langle n\rangle$. Applying Corollary \ref{coh2}  we get (\ref{eqh5}).
\end{proof}

\par

Now we are ready to present the proof of Theorem \ref{th1}.

{\it Proof of Theorem \ref{th1}.}
 The sufficiency part of this theorem is obvious. It suffices to prove the necessity part.

 Suppose $\phi$ is a complete linear map on $\G_n$ satisfying (\ref{h1}).
 If $t=1$, then by Lemma \ref{leh2} and Corollary \ref{coh2}, $\phi$ is an edge permutation.

Suppose $2\le t\le n-1$.
 Applying Lemma  \ref{hle5}, for each $i\in \langle n\rangle$, there exists $i'\in \langle n\rangle$ such that (\ref{eqh5}) holds.   Denote by
 $\sigma: \langle n\rangle \rightarrow \langle n\rangle$  the map such that
 $\sigma(i)=i'$. Since $\phi$ is bijective, $i'\ne j'$ whenever $i\ne j$. Hence, $\sigma$ is a permutation of $\langle n\rangle$. Applying Lemma \ref{hle5} we have
 $$\phi(G_{ij})=G_{i'j'}=G_{\sigma(i)\sigma(j)}~{\rm for~all~}i,j\in \langle n\rangle,$$
which means  $\phi$ is a vertex permutation.
\hspace{9.5cm} $\Box$

From Theorem \ref{th1} we have the following corollary.
\begin{corollary} \label{coh1}
		Let $n$ be  a positive integer.  Then a  linear map $\phi: \G_n\rightarrow \G_n$  satisfies
		\begin{equation}\label{eqco1}
		\alpha(\phi(G))=\alpha(G)~{for~all~} G\in \G_n
		\end{equation}
		if and only if $\phi$ is a vertex permutation.
\end{corollary}
\begin{proof}
	Note that the only graph in $\G_n$ with independence number one is $K_n$.  (\ref{eqco1}) implies that $\phi$ is complete and (\ref{h1}) holds for all $t\in\langle n\rangle$. Applying Theorem \ref{th1} we see that (\ref{eqco1})  holds if and only if $\phi$ is a vertex permutation.
\end{proof}

{\it Remark.}  Let $n$ and $t$ be positive integers such that $2\le t\le n$. Applying similar arguments with Theorem \ref{th1},
we can characterize the complete linear maps such that
 $$\omega(\phi (G))=t~{\rm\iff} ~\omega(G)=t~{\rm~for~all~}G\in \G_n.$$

\section{Proof of Theorem \ref{mth1}}

Now  we  present a result on linear maps on matrices, which is stronger than Theorem \ref{mth1}.

\begin{theorem}\label{mth2}
	Let $n$ and $t$ be positive integers with $2\le t\le n-1$, and let $H\in S_n^+$ with all off-diagonal entries nonzero.  Then  a linear map $\phi: S_n^+\rightarrow S_n^+$  satisfies $\phi(J_n-I_n)=H$ and
 \begin{equation}\label{eqhm2}
	\alpha(\phi(X))=t~\iff \alpha(X)=t {\quad\rm for~ all\quad}X\in S_n^+.
	\end{equation}if and only if there is a permutation matrix $P$   such that
	$$\phi(X)=H\circ (P^TXP){\quad\rm for~ all\quad}X\in S_n^+.$$
\end{theorem}

\begin{proof}
 The sufficiency is obvious. For the necessity, suppose $\phi$ is a linear map on $S_n^+$ satisfying $\phi(J_n-I_n)=H$ and  (\ref{eqhm2}). We define a map $\varphi: \G_n\rightarrow \G_n$ as
$$\varphi(X)=G(\phi(A(X)))\quad {\rm for~ all\quad} X\in \G_n.$$
Then for any pair of graphs $X,Y\in \G_n$, we have
$$\varphi(X\cup Y)=G(\phi(A(X\cup Y)))=G(\phi(A(X)))\cup G(\phi(A(Y)))=\varphi(X)\cup \varphi(Y).$$
Notice that $$\alpha(B)=\alpha(G(B)) \quad \rm{ for~ all}\quad  B\in S^+_n. $$
We have
$$\alpha(X)=\alpha(A(X))=\alpha(\phi(A(X)))=\alpha(G(\phi(A(X))))=\alpha(\varphi(X))$$
for all $X\in \G_n$.
Moreover,
since $\phi(J_n-I_n)=H$, we have
$$\varphi(K_n)=G(\phi(A(K_n)))=G(\phi(J_n-I_n))=G(H)=K_n.$$
Therefore, $\varphi$ is a completely linear map on $\G_n$ satisfying (\ref{h1}). By Theorem \ref{th1}, $\varphi$ is a vertex permutation.  Therefore, there is a permutation matrix $P$ such that
$$A(\varphi(X))=P^TA(X)P\quad {\rm for~all\quad} X\in \G_n.$$

Denote by $sgn(B)=A(G(B))$ the sign matrix of $B\in S_n^+$. Then for any matrix $B\in S_n^+$,  we have $B\equiv(b_{ij})=\sum_{i<j}b_{ij}K_{ij}$, where $K_{ij}=E_{ij}+E_{ji}$ with $E_{ij}$ being the 0-1 matrix with its $(i,j)$-entry equal to 1 and all other entries zero.
Since $B$ and $\phi(K_{ij})$ are all in $S_n^+$, we have
\begin{eqnarray*}
sgn(\phi(B))
&=&sgn\left(\sum_{i<j}b_{ij}\phi(K_{ij})\right)=sgn\left(\sum_{i<j, b_{ij}\ne 0}\phi(K_{ij})\right)\\
&=&sgn\left(\phi(\sum_{i<j, b_{ij}\ne 0}K_{ij})\right)=sgn(\phi(sgn(B))).
\end{eqnarray*}
Therefore,
$$sgn(\phi(B))=sgn(\phi(sgn(B)))=A(\varphi(G(B)))=P^TA(G(B))P=P^Tsgn(B)P.$$ It follows that there is a  matrix $R_B\in S_n^+$ such that
\begin{equation}\label{eq425}
\phi(B)=R_B\circ (P^TBP).
\end{equation}

Now we prove that the matrices $R_B$'s can be chose to be a uniform matrix. By (\ref{eq425}) we have $\phi(K_{ij})=r_{ij}P^TK_{ij}P$ with $r_{ij}$ being a positive number for all $i<j$. Let $R=(r_{ij})_{n\times n}\in S_n^+$. Then for every $B=(b_{ij})\in S_n^+$,
we have
$$\phi(B)=\sum_{i<j}b_{ij}\phi(K_{ij})=\sum_{i<j}b_{ij}r_{ij}P^TK_{ij}P=P^T(R\circ B)P=(P^TRP)\circ (P^TBP).$$

Note that $\phi(J_n-I_n)=H$. We have $P^TRP=H$ and
$$\phi(X)=H\circ (P^TXP)\quad {\rm for~all\quad}X\in S_n^+.$$
This completes the proof.
\end{proof}

\par

Now we are ready to present the proof of Theorem \ref{mth1}.

{\it Proof of Theorem \ref{mth1}.} The sufficiency is clear. We prove the necessity.  By (\ref{eqhm1}), we have $\alpha(\phi(X))=1$ whenever $\alpha(X)=1$. Since $\alpha(\phi(J_n-I_n))=\alpha(J_n-I_n)=1$, all off-diagonal entries of  $H=\phi(J_n-I_n)$ are positive. Applying Theorem \ref{mth2},  there is a permutation matrix $P$   such that
	$$\phi(X)=H\circ (P^TXP){\quad\rm for~ all\quad}X\in S_n^+.$$
This completes the proof. \hspace{12cm}$\Box$

\section*{Acknowledgement}
This work was  supported by the National Natural Science Foundation of China (No. 12171323), the Science and  Technology Foundation of Shenzhen City (No. JCYJ20190808174211224) and Guangdong Basic and Applied Basic Research
Foundation (No. 2022A1515011995).

\end{document}